\documentclass[12pt]{amsart}
\usepackage{amsmath}
\usepackage{amsthm}
\usepackage{amssymb,bbm}
\usepackage{stmaryrd}
\usepackage{enumerate}
\newcommand{\E}{\mathbbm{E}}
\newcommand{\R}{\mathbbm{R}}
\newcommand{\N}{\mathbb{N}}
\renewcommand{\H}{\mathrm{H}}
\renewcommand{\d}{\mathrm{d}}
\newcommand{\e}{\mathrm{e}}

\newcommand{\dz}{\mathrm{d}z}
\newcommand{\SPAN}{\mathrm{span}}
\newcommand{\proj}{\mathrm{P}}

\theoremstyle{plain}
\newtheorem{theorem}{Theorem}[section]
\newtheorem{lemma}[theorem]{Lemma}
\newtheorem{corollary}[theorem]{Corollary}
\newtheorem{proposition}[theorem]{Proposition}
\newtheorem{conjecture}[theorem]{Conjecture}

\theoremstyle{definition}

\theoremstyle{remark}
\newtheorem{remark}[theorem]{Remark}

\begin{document}

\author{O. Guédon, F. Souli} \address{} \email{}

\title{An extension of a theorem of Chevet.} \

\address{Univ Gustave Eiffel, Univ Paris Est Creteil, CNRS, LAMA UMR8050 F-77447 Marne-la-Vallée, France}
\email{olivier.guedon@univ-eiffel.fr, fabien.souli@edu.univ-eiffel.fr}

\begin{abstract}
We adopt a modern approach to extend 
a comparison inequality of Chevet for Gaussian processes to a broader class of functions.
Furthermore, we consider centered Gaussian random vectors associated with a family of $n+1$ vectors on the unit sphere 
$S^{n-1}$ and investigate configurations of the vertices for which these new functionals  
could be maximum  for the regular simplex.
\end{abstract}

\maketitle


\section{Introduction}
Gaussian random vectors play a significant role not only in probability and statistics but also in fields such as functional analysis and convex geometry. In these areas, many problems can be reformulated in terms of comparisons between the expectations of a function evaluated at two Gaussian random vectors. A typical example involves comparing the  expectations  $\E f(X)$ and $\E f(Y)$ for centered 
Gaussian random vectors $X, Y$ and a real-valued function $f$, under simple assumptions about the natural distances induced by the processes  $X$ and $Y$. 
In this article, we propose a modern approach to a result by Chevet \cite{Chevet1978} concerning the comparison of the expectations of the square of the projection of $\max_{1 \le i \le n} X_i$ onto the orthogonal complement of the reproducing kernel Hilbert space $H_X$. Along the way, we extend her result to a new class of functions. In the second part of the article, we investigate configurations of a polytope with $n+1$ vertices $u_1, \ldots, u_{n+1}$ on the unit sphere $S^{n-1}$ for which these quantities are maximized over the family of Gaussian processes defined by $X_i = \langle G, u_i \rangle$, where $G$ is a standard $\mathcal{N}(0, \mathrm{Id})$ Gaussian random vector in $\R^{n}$. This is in the same spirit as the simplex mean width conjecture.

\subsection{Notations}
We begin by introducing some notations that will be used throughout
this article.  Let $n\in \N$, and consider $\R^n$ equipped with the
canonical inner product~$\langle \cdot , \cdot \rangle$. The  Euclidean norm of a vector $x$ is denoted by $|x|$.
The all ones vector $(1, 1, \dots, 1)$ is represented by $\mathbbm{1}$ while $[n]$ refers to the set of integers from $1$ to $n$.
For a centered Gaussian  random vector $X$ of $\R^n$ with covariance matrix $\Sigma_X$,  we define for all
$i,j \in [n]$,
\[ 
\sigma_{ij}^X := \left[ \Sigma_X\right]_{ij} = \E \left[ X_i X_j \right] \qquad 
\mathrm{and}
\qquad
\d_{X}(i,j) := \left( \E \left[ (X_i-X_j)^2 \right] \right)^{1/2}.
\] 
The orthogonal projection
of $\mathrm{L}^{2}\left( \Omega, \mathbbm{P} \right)$ onto
$\H_X:=\mathrm{span}(X_1, X_2, \dots, X_n)$ (resp. $\H_X^{\perp}$)
will be denoted by $\proj_{\H_X}$ (resp. $\proj_{\H_X^\perp}$).

We denote by $G$  the standard $\mathcal{N} (0, \mathrm{Id})$ Gaussian random vector in
 $\R^n$. Throughout this work, we will refer to a particular Gaussian random vector, associated to a regular simplex with vertices on the unit sphere. This corresponds to a centered Gaussian random vector $S$ with covariance matrix defined, for all $i, j \in [n]$, by:
 \begin{equation} \label{gaussian_regular_simplex}
 \sigma_{ij}^{S} = 
    \begin{cases}
      1 & \textrm{if }  i=j \\
      -\frac{1}{n-1} & \textrm{if } i\neq j.
    \end{cases} 
\end{equation} 
For all $\beta>0$, we define the function $f_{\beta} :  \R^n \to  \R$ by 
\[
f_\beta(x) =  \frac{1}{\beta} \log\left( \sum_{k=1}^n \e^{\beta x_k} \right).
\]
It is a $C^\infty$ function that approximates the $\max$ function on $\R^n$ as $\beta$ goes to infinity, as shown by the following inequality: for
all $x \in \R^n$,
\begin{equation}
	\label{eq:approx}
 \max_{1 \le i \le n}(x_i) \leq f_{\beta}(x) \leq
\frac{\log(n)}{\beta}+ \max_{1 \le i \le n}(x_i). 
\end{equation}
Moreover, it is of moderate growth in the sense that for each $a > 0$, 
\[
\lim_{|x| \to + \infty} f_\beta(x) e^{-a |x|^2} = 0.
\]
All its  partial derivatives up to order 2 are also of moderate growth.

\subsection{Main results and organization of the paper}

For general background on Gaussian processes, we refer the reader to \cite{MarcusRosen2006, TalagrandSpin}. We consider
two Gaussian random vectors $X, Y$ as well as a function $f$ and we want to
compare $\E f(X) $ with $\E f(Y) $. When $f=\max$ and $X, Y$ are
centered, the following result is known as the Fernique-Sudakov inequality, see \cite{Fernique1974} where Fernique attributes this extension of Slepian's Lemma to Chevet while in \cite{BadrikianChevet1974}, it is attributed to Sudakov \cite{Sudakov1971}.
\begin{theorem}\cite{Fernique1974} \label{theorem:slepian} Let
$X, Y$ be two centered Gaussian random vectors in $\R^n$ such that for all $i,j \in [n], \d_X(i,j) \leq \d_Y(i,j).$ Then
\begin{equation}
	\label{eq:Slepian}
\E \left[ \max_{1 \le i \le n}(X_i) \right] \leq \E \left[ 
\max_{1 \le i \le n}(Y_i) \right]. 
\end{equation}
\end{theorem}
In \cite{slepian1962}, Slepian introduces the additional assumption that $\E \left[ X_i^2 \right] = \E  \left[ Y_i^2\right]$ and shows that the tail probability of $\max_{1 \le i \le n} X_i$ is smaller than that of $\max_{1 \le i \le n} Y_i$. 
Building on this result, one can derive Theorem \ref{theorem:slepian} with a multiplicative constant factor of 2 in \eqref{eq:Slepian}. 
Chevet \cite{Chevet1978} establishes a powerful consequence of these inequalities using a decoupling argument, in the  study of the operator norm of Gaussian random operators. This theory was later extended, with major contributions including Gordon's inequalities \cite{gordon1992} which provide comparison results for min-max functions. 
The methodology underlying all these proofs relies on an interpolation path from $X$ to $Y$, demonstrating that along this path, the expectation of the function is non-decreasing. This is achieved through a differentiation argument. 
However a key challenge arises as the max function is not  differentiable. To overcome this difficulty, the authors employed arguments based on the theory of distributions.  
Later (see e.g. \cite{LiQueffelec2004,LiQueffelec2018}), proofs of Gordon's inequalities were refined using a convolution argument which yields a smooth approximation of the max or min-max functions.
One of its more recent proofs uses explicit approximations of the max function and this is where the family of functions $\{f_\beta\}_{\beta>0}$ appears. In \cite{chatterjee}, it is shown that  if for all $i,j \in [n], \d_X(i,j) \leq \d_Y(i,j)$ then  for every $\beta >0$, 
$\E f_\beta (X) \le \E f_\beta (Y)$. A more sophisticated family of functions is introduced in \cite{peccati_turchi_2023} providing not only a new proof of Gordon's inequalities but also a stability result.

As noted in \cite{chevet1976}, Inequality \eqref{eq:Slepian} does not hold for squares. However, for a centered Gaussian random vector $X$, 
$\proj_{H_X} \left( \max_{1\leq i\leq n} (X_i) \right) \in \SPAN \{X_1, \ldots, X_n\}$ is also centered and
one has 
\[
\E \left[ \max_{1\leq i\leq n} (X_i) \right]
= \E\left[ \proj_{H_X^\perp} ( \max_{1\leq i \leq n} (X_i) ) \right].
\]
Thus, Theorem \ref{theorem:slepian} can be interpreted as a
comparison between the expectations of
$\displaystyle \proj_{H_X^\perp} ( \max_{1\leq i \leq n} (X_i) )$ 
and
$\displaystyle \proj_{H_Y^\perp} ( \max_{1\leq i \leq n} (Y_i) )$ for centered Gaussian random vectors $X, Y$.
From this perspective, Chevet \cite{chevet1976} established a
comparison inequality involving squares and extended Theorem \ref{theorem:slepian} to the following result.

\begin{theorem}\cite{chevet1976}\label{theorem:chevet}
  Let $X, Y$ be two centered Gaussian random vectors in $\R^n$ such that for all
$i,j \in [n], \d_X(i,j) \leq \d_Y(i,j).$ 
 Then
  \[ \E \left[ \left( \proj_{H_X^\perp}(\max_{1\leq i \leq n} (X_i)
      ) \right)^2 \right] \leq \E \left[ \left(
      \proj_{H_Y^\perp}(\max_{1\leq i \leq n} (Y_i) )
    \right)^2 \right].\]
\end{theorem}

Our first contribution consists in providing not only a new proof of Theorem \ref{theorem:chevet} by using
modern methods but also in extending the result to a new class of functions. The method of proof is based on the classical interpolation technique mentioned above, combined with
a differentiation argument that is easier to carry due to the assumptions on the partial derivatives of the functions.  The result is
stated as follows and will be proven in Section \ref{sec:proof}.

\begin{theorem}\label{theorem:main}
  Let $X, Y$ be two centered Gaussian random vectors in $\R^n$ such that  for all
$i,j \in [n],d_X(i,j) \leq \d_Y(i,j).$ 
Assume that $f : \R^n \to \R$ is a $C^2$ function, that $f$ and all its partial derivatives up to order 2 are of moderate growth, and that it satisfies the following properties: 
  \begin{enumerate}[(a)]
  \item \label{thm:a_somme}
    $\qquad \mathit{for \ all \ } x\in \R^n, \quad \mathit{for \ all \ } u \in \R, \qquad
    f(x+ u \mathbbm{1})=f(x)+ u, $
  \item \label{thm:b_positive}
    $ \qquad \mathit{for \ all \ } i \in [n], \qquad
    \partial_i f\geqslant 0, $
  \item \label{thm:c_negative}
    $ \qquad \mathit{for \ all \ } i \ne j \in [n], \quad  \partial_{ij}^2 f\leq 0, $
  \item \label{thm:d_max}
    $ \qquad \mathit{for \ all \ } x\in \R^n, \qquad f(x)\geqslant \displaystyle
    \max_{1\leq i \leq n } (x_i) ,$
  \item \label{thm:e_lambda}
    $ \qquad \mathit{there \ exists \ } \lambda \in \R, \mathit{\ such \ that \ for \ all \ } i \ne j \in 
    [n], \quad  \partial_i f\partial_j
    f +\lambda \, \partial^2_{ij} f \leq 0.$
  \end{enumerate}
  Then
  \[ \mathbb{E} \left[  \left( \proj_{H_X^\perp} (f(X)) +\lambda \right) ^2 \right]
    \leq \mathbb{E} \left[  \left( \proj_{H_Y^\perp} (f(Y)) +\lambda
    \right)^2 \right]. \]
\end{theorem}
Due to the simple form of the partial derivatives of $f_\beta$, the assumptions $(a)-(e)$ of Theorem \ref{theorem:main} are easily verified for $f_\beta$ by setting $\lambda = 1/ \beta$ in $(e)$. This holds for all $\beta >0$. 
A slight improvement shows the following result, which generalizes Chevet's comparison Theorem \ref{theorem:chevet}. 
\begin{corollary} 
\label{theorem:application_main_theorem}
Let $\mu \in [0,+\infty)^n$ and $X, Y$ be two centered Gaussian random
vectors in $\R^n$ such that for all
$i,j \in [n], \d_{X}(i,j) \leq \d_{Y}(i,j).$
Then, for all $\beta >0$ one has
\[
\E \left[ \left(\proj_{H_X^\perp}(f_{\beta}(X+\mu))+\frac{1}{\beta} \right)^2 \right]
\leq 
\E \left[ \left(\proj_{H_Y^\perp}(f_{\beta}(Y+\mu))+\frac{1}{\beta} \right)^2 \right]
\]
and in particular
\[
\E  \left[ \left( \proj_{H_X^\perp}( \max_{1\leq i \leq n}(X_i+\mu_i )) \right)^2 \right]
\leq 
\E \left[ \left( \proj_{H_Y^\perp}(\max_{1\leq i \leq n}( Y_i+\mu_i )) \right)^2  \right].
\]
\end{corollary}
The proof of this result is  provided in Section \ref{sub:appl}, where other functions satisfying the hypotheses of Theorem \ref{theorem:main} are also discussed.

A second objective of the paper is to propose a new problem concerning the maximization of $\E f(X)$ over Gaussian random vectors subject to a geometric constraint.
This type of optimization problem is classical in probability theory and has connections with convex geometry.  
Let $K$ be a convex compact set in $\R^n$, with $0$ as an interior point and let $\|\cdot\|_K$ and $h_K$ denote its gauge and support functions, both with respect to the origin. Let also $K^o$ be the polar of $K$ with respect to the origin:
$K^o = \{y, \langle x, y \rangle\le 1 \mathrm{ \ for \ all \ } x \in K\}$. An important quantity in the Asymptotic Geometric Analysis is the $\ell$-norm associated to $K$  defined by $\ell(K) = \E \|G\|_K$. 
The geometric mean width of $K$, $W(K)$, is defined by 
\[
W(K) = \int_{S^{n-1}} (h_K(\theta) + h_K(-\theta)) d\sigma(\theta) = 
2 \int_{S^{n-1}} h_K(\theta) d\sigma(\theta)
\]
where $\sigma$ is the Haar probability measure on the unit sphere $S^{n-1}$.
Mean width and $\ell$-norm are thus related by the following formula:
\[ 
2 \ell(K^o)= 2 \E \left[ \sup_{x \in K} \langle G, x \rangle \right] = c_n \sqrt{n} \, W(K)
\]
where $c_n = \ell(B_2^n)/\sqrt{n} \to 1$ as $n$ goes to infinity. 
There are various geometric problems that investigate the shapes of convex sets maximizing the mean width among convex compact sets with specific geometric structure. For example, it has been known since the work of Barthe \cite{Barthe1998}
that, among  convex bodies $K$ in L\"owner position (those for which the ellipsoid of minimal volume containing $K$ is the unit Euclidean ball), 
the regular simplex with vertices on the unit sphere  has maximal mean width (see also \cite{SS} for the symmetric case). 

Another example involves  using the function
$f(x)=\max_{1\leq i \leq n} \lvert x_i \rvert$  
and maximizing the expression $\E f(X)$ among Gaussian random vectors in $\R^n$ with  unit variance coordinates.
This problem admits a simple solution, which follows from 
the Gaussian correlation inequality, proved by Royen \cite{royen2014}.
\begin{theorem}\cite{royen2014}
For every centered Gaussian random vector $X$ in $\R^n$ such that  $\E \left[ X_i^2  \right] = 1$ for all $i \in [n]$, one has
\[
 \E\left[  \max_{1\leq i \leq n} \lvert X_i \rvert \right]
\leq
\E\left[ \max_{1\leq i \leq n} \lvert G_i \rvert \right]
\]
where $G$ is a standard $\mathcal{N} (0, \mathrm{Id})$ in $\R^n$.
\end{theorem}
Geometrically, this implies that among all families of $n$-tuples of points $\{u_1, \ldots, u_n\}$ on the unit sphere $S^{n-1}$, the mean width of 
the symmetric convex hull of $\{u_1, \ldots, u_n\}$ is maximized when the points $\{u_1, \ldots, u_n\}$ form an orthonormal basis of $\R^n$.

However, for the function $f(x)= \max_{1 \le i \le n} x_i$, the problem remains open and there exists a
conjecture.
\begin{conjecture}
	\label{Conjecture}
For every centered Gaussian random vector $X$ in $\R^n$ such that  $\E \left[ X_i^2 \right] = 1$ for all $i \in [n]$, one has
\begin{equation}
	\label{eq:simplexmeanwidthconj}
\E \left[ \max_{1\leq i \leq n} X_i \right]
\leq
\E \left[ \max_{1\leq i \leq n} S_i \right] 
\end{equation}
where we recall that $S$ is a centered Gaussian random vector with covariance matrix defined by~(\ref{gaussian_regular_simplex}).
\end{conjecture}
Geometrically, the conjecture asserts that among all $n$-tuples of points 
\linebreak
$\{u_1, \ldots, u_{n}\}$ on the unit sphere $S^{n-2}$,  the mean width of their convex hull  is maximized when the points form the vertices of a regular simplex. Indeed, let $\tilde{G}$ be
a standard $\mathcal{N}(0, \mathrm{Id})$ Gaussian random vector in $\R^{n-1}$ and define the Gaussian random vector $X$ in $\R^n$ by $X_i = \langle \tilde{G}, u_i \rangle$. When $\{u_1, \ldots, u_{n}\}$ are the vertices of a regular simplex in $\R^{n-1}$, the associated Gaussian random vector is $S$ and its covariance  is given by the equation \eqref{gaussian_regular_simplex}. Therefore Conjecture \ref{Conjecture} is the probabilistic form of the famous Simplex Mean
Width Conjecture. 
In \cite{klz2015}, the authors provide a detailed presentation of various equivalent probabilistic formulations of the conjecture. They prove an asymptotic version of the comparison inequality as $n$ goes to infinity (see Theorem 1.2 in \cite{klz2015}) as well as a particular case of the inequality 
\eqref{eq:simplexmeanwidthconj} (see Theorem 2.1 in \cite{klz2015}).
To our knowledge, these are the best known mathematical results regarding this conjecture and
we refer the reader to \cite{litvak2018} for a detailed presentation and history of the problem.

The new problem  we propose considers another probabilistic setting in which  the Gaussian random vector $S$ could be a maximizer.
\begin{conjecture}\label{conjecture:chevet_max_expectation}
For every centered Gaussian random vector $X$ in $\R^n$ such that  $\E \left[ X_i^2 \right] = 1$ for all $i \in [n]$, for all $\beta > 0$, one has
\[
\E \left[ \left( \proj_{H_X^\perp}   (f_{\beta}(X))+\frac{1}{\beta} \right)^2 \right]
\leq 
\E \left[ \left(\proj_{H_S^\perp} (f_{\beta}(S))+\frac{1}{\beta} \right)^2 \right]
\]
and in particular
\[ 
\E \left[ \left( \proj_{H_X^\perp} (\max_{1\leq i \leq n}(X_i)) \right)^2 \right]
\leq 
\E \left[ \left( \proj_{H_S^\perp}(\max_{1\leq i \leq n}(S_i)) \right)^2 \right].
\]
\end{conjecture}
We do not know a clear interpretation of this conjecture within the framework of convex geometry. 
Following the approach of~\cite{klz2015} and
using Theorem~\ref{theorem:main}, we provide a proof of this comparison inequality in a particular case.
\begin{theorem}\label{theorem:chevet_max_expectation_particular_case}
  Let $b_1, b_2, \ldots, b_n \in [-1,1]$ and let $X$ be a centered Gaussian random
  vector in $\R^{2n}$ with covariance matrix 
  \[ \begin{pmatrix}
    A_1 & 0 & 0 & \dots & 0 \\
    0 & A_2 & 0 & \dots & 0 \\
    0 & 0 & A_3 & \dots & 0 \\
    \vdots & &  &\ddots & \\
    0 & 0 & \dots & 0 & A_n
  \end{pmatrix} \in \mathcal{M}_{2n}(\mathbb{R})\]
where $A_i = \begin{pmatrix} 1 & b_i \\ b_i & 1 \end{pmatrix}$. Then, for all $\beta>0$, one has

\[ \E \left[ \left( \proj_{H_X^\perp}(f_{\beta} (X))+\frac{1}{\beta}
  \right)^2 \right] \leq \E \left[ 
  \left(\proj_{H_S^\perp}(f_{\beta}(S))+\frac{1}{\beta} \right)^2 \right]
\]
and in particular
\[ \E  \left[ \left( \proj_{H_X^\perp}\left(\max_{1\leq i \leq 2n}
      (X_i)\right) \right)^2  \right]
\leq \E \left[ \left(
    \proj_{H_S^\perp}\left(\max_{1\leq i \leq 2n}
      (S_i)\right) \right)^2 \right]. 
\]
\end{theorem}
The proof will be presented in Section \ref{sec:application} and is divided into two steps. First, using Theorem~\ref{theorem:main}, we show that it suffices to
establish  the inequality when all the $b_i$'s are equal to $-1$. Following \cite{klz2015}, and due to the particular structure of the functions, we can replace the target $S$ by a scalar multiple of a Gaussian random vector with independent entries. The specific form of the partial derivatives of the function $f_\beta$, combined with the symmetries of the covariance matrix, allows us to prove that the studied quantity is non decreasing along the new path.

\section{The comparison Theorem via an interpolation formula}
\label{sec:proof}

\subsection{An expression of the projection.}
We recall the classical integration by parts formula. For a centered Gaussian random vector $X$, for 
any $C^1$  function $g$ such that $g$ and its partial derivatives are of moderate growth, for any $i \in [n]$, one has
\begin{equation}
	\label{eq:IPP}
\E \left[ X_i g(X)\right] = \sum_{\ell=1}^{n} \E \left[ X_{\ell} X_i \right] \, \E \left[ \partial_\ell g(X) \right].
\end{equation}
We will use it to give a simple expression of $\proj_{H_X} (f(X))$ for a
centered Gaussian random vector.
\begin{lemma} \label{lemma:expression_projection} 
Let $X$ be a centered Gaussian random vector with a positive definite covariance
matrix. Assume that $f : \R^n \to \R$ is a  $C^1$  function and that $f$ and its first order partial derivatives are of moderate growth, then
\[ 
\proj_{H_X}(f(X)) = \sum_{k=1}^n X_k  \mathbb{E} \left[ 
\partial_k f(X) \right].
\]
\end{lemma}
\begin{proof}
Let $\alpha \in \R^n$ be such that
\[
\proj_{H_X}(f(X)) = \sum_{i=1}^n \alpha_i X_i.
\]
By definition of the orthogonal projection,
for all $i_0\in [n]$, 
\[
 \mathbb{E}\left[ f(X)X_{i_0} \right]= \E \left[ \proj_{H_X}(f(X))X_{i_0} \right]
      =\sum_{i=1}^n \alpha_i\mathbb{E}\left[ X_iX_{i_0} \right]
       =[\Sigma_X.\alpha]_{i_0}.
\]
Using the Gaussian integration by parts formula \eqref{eq:IPP}, we get
  \[
      \mathbb{E}\left[ f(X)X_{i_0} \right] =\sum_{i=1}^n \mathbb{E} \left[ \partial_i f(X) \right] \mathbb{E} \left[ X_iX_{i_0}\right] \\
                                            =[\Sigma_X. \mathbb{E}\left[\nabla f(X) \right] ]_{i_0}. 
  \]
  Hence,
  $\Sigma_X \cdot\alpha = \Sigma_X \cdot \mathbb{E}\left[\nabla f(X) \right]$
  and we conclude that $\alpha = \mathbb{E}\left[\nabla f(X) \right]$
  due to the invertibility of $\Sigma_X$.
\end{proof}

\subsection{The interpolation path formula.}
We use the traditional method of interpolation between $X$ and $Y$  along a monotonic
path to compare  the quantities of interest.

\begin{proposition}\label{proposition:interpolation}
Let $X, Y$ be two independent centered Gaussian random vectors with positive definite
covariance matrices. Let $f$ be a $C^2$  function such that $f$ and its partial derivatives up to order 2 are of moderate growth, and
such that for all $x \in \R^n$ and all $u \in \R$,
\begin{equation} 
	\label{eq:somme}
f(x + u\mathbbm{1}) = f(x)+u. 
\end{equation} 
For all $t\in [0,1]$, let
\[  
  Z_t =\sqrt{1-t} \, X + \sqrt{t} \, Y
\mathrm{ \ and \ }
 \varphi(t) = \mathbb{E} \left[ ( \proj_{H_{Z_t}^\perp}(f(Z_t)) )^2 \right].
\]
Then, one has
\[
\begin{aligned}
\varphi'(t) = 
\frac{1}{2} \sum_{i\neq j} M_{ij} 
 &
\bigg( \mathbb{E}\left[ \partial_if(Z_t) \partial_j f(Z_t) \right] -
\mathbb{E}\left[ \partial_i f(Z_t) \right] \mathbb{E} \left[ \partial_j f(Z_t)\right] 
\\
+ & \
\mathbb{E} \left[ \proj_{H_{Z_t}^\perp}(f(Z_t))\partial^2_{ij} f(Z_t) \right] 
\bigg)
\end{aligned}
\]
where $M_{ij} = \d_X(i,j)^2-\d_Y(i,j)^2$.
\end{proposition}

\begin{proof}
Let $t\in [0,1]$. By lemma~\ref{lemma:expression_projection}, one has
\begin{equation}
	\label{eq:firstformula}
\proj_{H_{Z_t}^\perp} (f(Z_t)) = f(Z_t)-\sum_{i=1}^{n} (Z_{t})_{i} \E \left[ \partial_{i} f(Z_t) \right] 
                               = f(Z_t) - \langle Z_t, \E \left[  \nabla f(Z_t) \right] \rangle. 
\end{equation}
Therefore
\[
\frac{\d}{\d t}\left( \proj_{H_{Z_t}^\perp}(f(Z_t))\right)
= \left \langle  \frac{\d}{\d t} Z_t, \nabla f(Z_t) - \E \left[ \nabla f(Z_t) \right] \right \rangle
- \left \langle  Z_t, \frac{\d}{\d t} \E \left[ \nabla f(Z_t) \right] \right \rangle
\]
and 
\[
\varphi'(t) = 2 \E  \left[ \proj_{H_{Z_t}^\perp}(f(Z_t)) 
	   \left(  \langle  \frac{\d}{\d t} Z_t, \nabla f(Z_t) - \E \left[  \nabla f(Z_t) \right]  \rangle
	 - \langle  Z_t, \frac{\d}{\d t} \E \left[ \nabla f(Z_t) \right] \rangle \right) \right].
\]
By definition of the orthogonal projection, 
\[
\E \left[  \proj_{H_{Z_t}^\perp}(f(Z_t)) Z_t \right] = 0
\]
hence
\[
\varphi'(t) =  2 \, \E \left[  \proj_{H_{Z_t}^\perp}(f(Z_t))
 \left\langle  \frac{\d}{\d t} Z_t, \nabla f(Z_t) - \E \left[  \nabla f(Z_t) \right]  \right\rangle \right].
\]
Moreover,
\[
\frac{\d}{\d t} Z_t = \left( -\frac{1}{2\sqrt{1-t}}X +\frac{1}{2\sqrt{t}} Y \right).
\]  
Therefore
\[ 
\begin{aligned}
\varphi'(t)&=-\frac{1}{\sqrt{1-t}} \sum_{i=1}^n \E \left[  X_i \proj_{H_{Z_t}^\perp}(f(Z_t)) \left( \partial_i f(Z_t) - \E \left[  \partial_i f(Z_t) \right]  \right) \right] \\
           &\quad +\frac{1}{\sqrt{t}} \sum_{i=1}^n \E \left[  Y_i \proj_{H_{Z_t}^\perp}(f(Z_t)) \left( \partial_i f(Z_t) - \E \left[  \partial_i f(Z_t) \right]   \right) \right].
  \end{aligned} 
\]
Using \eqref{eq:firstformula}, for all $i, j \in [n]$, one has
  \begin{align*}
    \frac{1}{\sqrt{1-t}  }\partial_{x_j} \left(  \proj_{H_{Z_t}^\perp}(f(Z_t)) \left( \partial_i f(Z_t) - \E \left[ \partial_i f(Z_t) \right]  \right)  \right)
    & = 
    \\
    (\partial_j f(Z_t) - \E \left[ \partial_j f(Z_t) \right]  ) (\partial_i f(Z_t) - \E \left[ \partial_i f(Z_t) \right] ) 
    & +    \proj_{H_{Z_t}^\perp}(f(Z_t))\partial_{ij}^2 f(Z_t)  
    \end{align*}
and
\begin{align*}
 \frac{1}{\sqrt{t}} \partial_{y_j} \left(  \proj_{H_{Z_t}^\perp}(f(Z_t)) \left( \partial_i f(Z_t) - \E \left[ \partial_i f(Z_t) \right] \right)  \right)
  & = 
  \\
     (\partial_j f(Z_t) - \E \left[ \partial_j f(Z_t) \right]  ) (\partial_i f(Z_t) - \E \left[ \partial_i f(Z_t) \right]  ) 
  & +  \proj_{H_{Z_t}^\perp}(f(Z_t))\partial_{ij}^2 f(Z_t). 
    \end{align*}
By the independence of $X$ and $Y$,  and using the Gaussian
integration by parts formula \eqref{eq:IPP}, we obtain 
\begin{align*}
\varphi'(t)&= -\frac{1}{\sqrt{1-t}} \sum_{i,j} \sigma_{ij}^X \E \left[ \partial_{x_j} \left(  \proj_{H_{Z_t}^\perp}(f(Z_t))\left( \partial_i f(Z_t) - \E  \left[ \partial_i f(Z_t) \right]  \right)   \right) \right] \\
 &\quad +\frac{1}{\sqrt{t}} \sum_{i,j} \sigma_{ij}^Y \E \left[ \partial_{y_j} \left(  \proj_{H_{Z_t}^\perp}(f(Z_t))\left( \partial_i f(Z_t) - \E \left[  \partial_i f(Z_t)  \right] \right)   \right) \right]
\end{align*}
from which we conclude that
\begin{equation} 
\label{eq:derivative}
 \begin{split}
   \varphi'(t)=\sum_{1\leq i,j \leq n} \left( \sigma_{ij}^Y-\sigma_{ij}^X\right) \Big( &
   \E \left[ \partial_i f(Z_t) \partial_j f(Z_t) \right] - \E \left[ \partial_i f(Z_t)\right] \E \left[ \partial_j f(Z_t) \right] 
   \\
  &+ \E \left[   \proj_{H_{Z_t}^\perp} ( f(Z_t)) \partial_{ij}^2 f(Z_t) \right]   \Big).
\end{split}
\end{equation}
For all $i,j \in [n], \sigma_{ij}^Y-\sigma_{ij}^X = \frac{1}{2} \left( M_{ij} - N_i-N_j \right)$ 
where 
$N_i= \sigma_{ii}^X-\sigma_{ii}^Y$ and
$N_j= \sigma_{jj}^X-\sigma_{jj}^Y$.
By \eqref{eq:somme}, for all
$x\in \R^n$, the function $u\in \R \mapsto f(x+u\mathbbm{1})-u$ is
constant. Its derivative is therefore equal to zero and differentiating the relation with respect to $x_j$ shows the following identities:
\[ 
\sum_{i=1}^n \partial_i f = 1, \quad \mathrm{and} \quad
\mathrm{for \ all \ } j\in [n], \quad \sum_{i=1}^n
\partial_{ij}^2 f=0.
\]
This implies that
\[ 
\begin{aligned}
\sum_{1\leq i,j \leq n} N_j \Big( &
\E \left[ \partial_i f(Z_t) \partial_j f(Z_t) \right] - \E \left[ \partial_i f(Z_t)\right] \E \left[ \partial_j f(Z_t) \right] 
\\
&+ \E \left[   \proj_{H_{Z_t}^\perp} ( f(Z_t)) \partial_{ij}^2 f(Z_t) \right]   \Big) = 0
\end{aligned}
\]
 and the same with $N_i$ instead of $N_j$. By \eqref{eq:derivative}, we conclude that
 \[
 \begin{aligned}
 	\varphi'(t) = 
 	\frac{1}{2} \sum_{i\neq j} M_{ij} 
 	\Big( &
 	\E \left[ \partial_i f(Z_t) \partial_j f(Z_t) \right] - \E \left[ \partial_i f(Z_t)\right] \E \left[ \partial_j f(Z_t) \right] 
 	\\
 	&+ \E \left[   \proj_{H_{Z_t}^\perp} ( f(Z_t)) \partial_{ij}^2 f(Z_t) \right]   \Big)
 \end{aligned}
 \]
 since $M_{ii} = 0$ for all $i=1, \ldots, n$.
\end{proof}

\subsection{Proof of theorem \ref{theorem:main}}
First of all we note that $X$ and $Y$ can be replaced by independent
copies. Moreover, by approximation, we can also assume that $\Sigma_X$
and $\Sigma_Y$ are positive definite. Using the same notation as in Proposition \ref{proposition:interpolation} and defining by $\tilde{f}$ the function $f+\lambda$, we see by Lemma \ref{lemma:expression_projection} that for all $t \in [0,1],$
\[
\proj_{H_{Z_t}^\perp} ( \tilde{f}) (Z_t) = \proj_{H_{Z_t}^\perp} (f(Z_t))  +\lambda  
= f(Z_t) - \langle Z_t, \E \nabla f (Z_t) \rangle + \lambda.
\] 
Let $\tilde{\varphi}$ be the function defined on $[0,1]$ by
\[ 
\tilde{\varphi} (t) =\E \left[ \left(  \proj_{H_{Z_t}^\perp} ( \tilde{f}) (Z_t)  \right)^2 \right] =  \E \left[ \left(  \proj_{H_{Z_t}^\perp} (f(Z_t)) +\lambda  \right)^2 \right].
\]
Since $f$ satisfies \eqref{thm:a_somme}, it is clear that assumption \eqref{eq:somme} is verified for $\tilde{f}$ and we can apply Proposition \ref{proposition:interpolation} to deduce that 
\begin{align}
	\label{eq:tildephi}
\tilde{\varphi} '(t)  &=\frac{1}{2} \sum_{i\neq j} M_{ij} 
		\Big( \E \left[  \partial_i f(Z_t) \partial_j f(Z_t) \right] -\E \left[ \partial_i f(Z_t)\right] \E \left[ \partial_j f(Z_t) \right]
		\\
		\nonumber
		& \qquad \qquad + \ \lambda \, \E \left[  \partial_{ij}^2 f(Z_t) \right] + \E \left[ \proj_{H_{Z_t}^\perp}(f(Z_t))\partial^2_{ij} f(Z_t) \right] \Big).
\end{align}
By assumption \eqref{thm:e_lambda} and \eqref{thm:b_positive}, we know that for all $i \ne j$, 
\begin{equation}
	\label{eq:terme1}
\E  \left[ \partial_i f(Z_t) \partial_j f(Z_t) \right] +  \lambda \, \E  \left[ \partial_{ij}^2 f(Z_t) \right] \le 0,
\quad
-\E \left[ \partial_i f(Z_t) \right] \E \left[ \partial_j f(Z_t) \right] \le 0.
\end{equation}
By assumption \eqref{thm:a_somme}, for all $x \in \R^n$, the function $u \mapsto f(x+u\mathbbm{1})-u$ is constant and taking its derivative, we get
\[
\sum_{i=1}^{n} \partial_{i}(f) = 1.
\]
Combining with assumption \eqref{thm:b_positive} and  Lemma \ref{lemma:expression_projection}, we deduce that
\[
\proj_{H_{Z_t}}(f (Z_t)) = \sum_{i=1}^{n} (Z_t)_i \E \left[ \partial_i f(Z_t) \right]  \le \max_{1 \le i \le n} (Z_t)_i.
\]
Asssumption \eqref{thm:d_max} allows to conclude that
\[
\proj_{H_{Z_t}^\perp}(f(Z_t)) = f(Z_t) - \proj_{H_{Z_t}}(f(Z_t)) \ge 0
\]
and lastly, by assumption \eqref{thm:c_negative}, we get that for all $i \ne j$,
\begin{equation}
	\label{eq:terme2}
\proj_{H_{Z_t}^\perp}(f(Z_t))\partial^2_{ij} f(Z_t) \le 0.
\end{equation}
Since for all $i \ne j, M_{ij} \le 0$, combining \eqref{eq:terme1} and \eqref{eq:terme2} with \eqref{eq:tildephi}, we conclude that 
for all $t \in [0,1],$
$\tilde{\varphi}'(t) \ge 0$. Thus,  $\tilde{\varphi}(0) \le \tilde{\varphi}(1)$ which is the conclusion of Theorem \ref{theorem:main}.
\qed
\begin{remark} \label{remark:main_theorem} The conclusion of 
Theorem \ref{theorem:main} can be reformulated as
\begin{eqnarray}
	\label{eq:other}
	\mathbb{E} \left[  f(X)^2 \right] +2\lambda \mathbb{E} f(X)- \langle \Sigma_X\mathbb{E}\left[\nabla f(X) \right] , \mathbb{E}\left[\nabla f(X) \right] \rangle 
	& 
	\\
	\nonumber
	\leq \ \mathbb{E} \left[ f(Y)^2 \right]  +2\lambda
		\mathbb{E} f(Y) - \langle\Sigma_Y\mathbb{E}\left[ \nabla f(Y) \right],
		\mathbb{E}\left[ \nabla f(Y) \right] \rangle.
\end{eqnarray}
\end{remark}
\noindent
To prove it, note that by definition of the orthogonal projection and by Lemma ~\ref{lemma:expression_projection}, one has
$\E \left[ f(X) \proj_{H_X}(f(X)) \right] = \E \left[ (\proj_{H_X}(f(X)))^2 \right]$ and  $\E\left[\proj_{H_X}(f(X))\right]=0$.
Hence
\begin{align*}
\E \left[ \left(  \proj_{H_X^\perp}(f(X))+\lambda \right)^2 \right]
& = \E \left[ \left(  f(X)-\proj_{H_X}(f(X))+\lambda  \right)^2 \right]
\\
& = \E \left[ \left(  f(X) +\lambda  \right)^2 \right] - \E \left[  (\proj_{H_X}(f(X)))^2 \right].
\end{align*}
Moreover, by Lemma ~\ref{lemma:expression_projection}, 
\[
\E \left[ \left(\proj_{H_X}(f(X)) \right)^2 \right] = \langle\Sigma_X\E\left[\nabla f(X)\right],
\E\left[ \nabla f(X)\right] \rangle.
\] 
The same argument holds for $Y$ and it shows that the conclusion of Theorem \ref{theorem:main} is equivalent to \eqref{eq:other}.

\subsection{Application}
\label{sub:appl}
Recall that the function $f_\beta$,  defined for all $x \in \R^n$ by
\[
f_{\beta}(x) = \frac{1}{\beta} \log\left( \sum_{k=1}^n \e^{\beta x_k} \right),
\]
is $C^{\infty}$, is a good approximation of $\max_{1 \le i \le n} (x_i$) by \eqref{eq:approx},  and that its partial derivatives may be easily computed:
\begin{equation}
		\label{eq:derivees}
		\mathrm{for \ all \ } i \in [n], \partial_{i}f_{\beta} = p_i,
		\qquad 
		\mathrm{and,}
		\qquad
		\mathrm{for \ all \ } i \ne j, \ \partial^{2}_{ij} f_{\beta} = - \beta p_i p_j
\end{equation}
where for all $x \in \R^n$ and $i \in [n],$
\[
p_i(x) =  \frac{\e^{\beta x_i}}{\sum_{k=1}^n \e^{\beta x_k}}.
\]
It is therefore easy to construct a family of functions that will satisfy the assumptions of Theorem \ref{theorem:main}.
\begin{lemma}
Let $\mu \in [0,+\infty)^n$ and $\beta>0$. We define the function
${h_{\beta}}$ for every $x \in \R^n$ by $ h_{\beta}(x)=f_{\beta}(x+\mu).$ 
Then $h_{\beta}$ satisfies assumptions \eqref{thm:a_somme}, \eqref{thm:b_positive}, \eqref{thm:c_negative}, \eqref{thm:d_max} and \eqref{thm:e_lambda} of Theorem~\ref{theorem:main} with $\lambda=\frac{1}{\beta}$.
\end{lemma}
\begin{proof}
By the expression of $f_\beta$, it is clear that $h_{\beta}(x+u \mathbbm{1})=h_{\beta}(x)+ u$ for all
$x\in \R^n$ and $u\in \R$. 
From \eqref{eq:derivees}, for all $i\ne j  \in [n],$ and all $x\in \R^n$
 \[
   \begin{aligned}
        & \partial_{i} h_{\beta} (x) =p_i(x+\mu)\ge 0, \\
        & \partial_{ij}^2 h_{\beta}(x) =-\beta p_i(x+\mu)p_j(x+\mu)\leq 0, \\
        & \partial_{i} h_{\beta} (x)\partial_{j} h_{\beta} (x)+\frac{1}{\beta} \partial_{ij}^2 h_{\beta}(x) =0
  \end{aligned}
\]
so that \eqref{thm:a_somme}, \eqref{thm:b_positive}, \eqref{thm:c_negative} and \eqref{thm:e_lambda} are satisfied. Since for all $i \in \{1, \ldots, n\}$, $\mu_i \ge 0$, it is clear that $h_{\beta}(x) \ge \max_{1 \le i \le n}(x_i + \mu_i) \ge \max_{1 \le i \le n} x_i$ and 
\eqref{thm:d_max} is satisfied.
\end{proof}
Corollary \ref{theorem:application_main_theorem} is proven by applying Theorem \ref{theorem:main} to the function $h_{\beta}$. And the moreover part follows by sending $\beta$ to $+\infty$.
\qed

It can be observed that for all $x \in \R^n$, 
\[
f_\beta(x) = \sup_{\theta \in \Delta_n} \left( \langle \theta, x \rangle - \frac{1}{\beta} \sum_{i=1}^n \theta_i \log(\theta_i) \right)
\]
where 
\[
\Delta_n := \{ \theta=(\theta_1, \ldots, \theta_n) \in (\R_+)^n \mathrm{\ such \ that \ } \sum_{i=1}^n \theta_i = 1  \}
\]
is a $n-1$-dimensional regular simplex in $\R^{n}$.
Indeed, the function 
\[
\theta \mapsto \langle \theta, x \rangle - \frac{1}{\beta} \sum_{i=1}^n \theta_i \log(\theta_i)
\]
is strictly concave on $(\R_+)^n$ and
 using Lagrange multipliers, we get that for all $x \in \R^n$, the supremum over $\Delta_n$ is attained at the unique point $\theta^*$ whose coordinates are
\[
\theta^*_i = \frac{e^{\beta x_i}}{\sum_{k=1}^{n} e^{\beta x_k}}. 
\]
It is possible to generalize this situation as follows. 
\begin{proposition}
	\label{prop:Omer}
Let $\psi: [0,1] \to \R$ be a  $C^2$  function such that for all $\alpha \in [0,1]$, $0 \le \psi''(\alpha) \le M$ and $(n-1) \psi(0) + \psi(1) = 0$.
Let $\Delta_n$ be the regular simplex defined by
\[
\Delta_n := \{ \theta=(\theta_1, \ldots, \theta_n) \in (\R_+)^n \mathrm{\ such \ that \ } \sum_{i=1}^n \theta_i = 1  \}
\]
and let $F_\psi: \R^n \to \R$ be given, for all $x \in \R^n$, by
\begin{equation}
	\label{eq:defF}
	F_{\psi}(x) = \sup_{\theta \in \Delta_n} \left( \langle \theta, x \rangle - \frac{1}{\beta} \sum_{i=1}^n ( \theta_i \log(\theta_i) + \psi(\theta_i) ) \right).
\end{equation}
Then, for all $\beta>0$, the function $F_\psi$ satisfies conditions $(a)-(e)$ of Theorem \ref{theorem:main} with $\lambda = (M+1)^2/\beta$.
\end{proposition}
\begin{proof}
By definition, for all $\theta \in \Delta_n$, $\langle \theta, \mathbbm{1} \rangle = 1$. Hence, $F_\psi$ satisfies condition $(a)$ 
of Theorem \ref{theorem:main}. 
For each $k \in [n]$, we select $\theta^{(k)}$ as the vector in $\R^n$ with all coordinates equal to zero, except for the $k^{th}$-coordinate,
which is set to $1$. We evaluate at this point $\theta^{(k)} \in \Delta_n$ and since $(n-1) \psi(0) + \psi(1) = 0$, we get that
$F_\psi(x) \ge x_k$. This shows condition $(d)$ of Theorem \ref{theorem:main}.

Handling the conditions on the derivatives of $F$ is more delicate.
This is why we made a specific choice of function in the supremum.
Since $\psi'' \ge 0$ on $[0,1]$, the function $\phi : \alpha \mapsto \alpha \log (\alpha) + \psi(\alpha)$ is strictly convex on $[0,1]$ and $\phi'' > 0$ on $(0,1)$. Moreover $\phi'$ is increasing on $(0,1)$ and is bijective from $(0,1)$ to $(-\infty, 1 + \psi'(1))$. We denote by $(\phi')^{-1}$ its reciprocal.
By strict convexity, the supremum in \eqref{eq:defF} is attained at a unique point $\theta^*$ in the interior of $\Delta_n$. 
From Danskin's theorem (see \cite{Danskin1967}, and \cite[Theorem 10.31]{Rockafellar}), the function $F$ is differentiable at every point $x \in \R^n$ and 
for all $i \in [n]$, $\partial_i F_\psi = \theta_i ^* \ge 0$. Thus, condition $(b)$ of Theorem \ref{theorem:main} is satisfied. 
Moreover, using Lagrange multipliers and the implicit function theorem, we get that there exists a differentiable function $c : \R^n \to \R$ such that for all $x \in \R^n$ and for all $i \in [n],$
\[
\theta_i^* =  (\phi')^{-1} (\beta x_i - c(x)) 
\qquad
\mathrm{and}
\qquad
\sum_{k=1}^{n} (\phi')^{-1} (\beta x_k - c(x)) = 1.
\]
By differentiation, for all $j \in [n]$, we obtain
\[
\partial_j c(x) \sum_{k=1}^{n} \frac{1}{\phi''(\theta_k^*)} = \frac{\beta}{\phi''(\theta_j^*)}.
\]
From this, we deduce that for all $j \ne i$,
\[
\partial^2_{ji} F_\psi = 
\partial_j \theta_i^* = - \partial_j c(x) \frac{1}{\phi''(\theta_i^*)}
= - \frac{\beta}{ \phi''(\theta_j^*) \phi''(\theta_i^*) \sum_{k=1}^{n} \frac{1}{\phi''(\theta_k^*)}}.
\]
Thus, condition $(c)$ is satisfied. Computing the second derivative of $\phi$, we get that for all 
$\alpha \in (0,1), \alpha \phi''(\alpha) = 1 + \alpha \psi''(\alpha)$ from which we deduce that $\phi''(\alpha) \ge 1/\alpha$ and $\alpha \phi''(\alpha) \le (M +1)$. 
Thus, a simple computation shows that 
\[
(\partial_i F_{\psi}) (\partial_j F_{\psi}) + \lambda \partial^2_{ij} F_\psi
= 
\theta_i^* \theta_j^* - \frac{(M+1)^2}{\phi''(\theta_j^*) \phi''(\theta_i^*) \sum_{k=1}^{n} \frac{1}{\phi''(\theta_k^*)}}
\le 0
\]
since $\sum_{k=1}^{n} \frac{1}{\phi''(\theta_k^*)} \le \sum_{k=1}^n \theta_k^* \le 1$. This shows condition $(e)$.
\end{proof}

\section{Proof of Theorem
    \ref{theorem:chevet_max_expectation_particular_case}}
\label{sec:application}
Let $U, X$ be two centered Gaussian random vectors such that
\begin{equation}
	\label{def_X}
\Sigma_{U} = \begin{pmatrix}
 A_1 &  0   & \dots & 0 \\
 0   &  A_2 & \dots & 0 \\
 \vdots &   &\ddots & \\
 0 & \dots & 0 & A_n
\end{pmatrix},  
\quad    
\Sigma_{X} = \begin{pmatrix}
A & 0 & \dots & 0 \\
0 & A  & \dots & 0 \\
\vdots &  &\ddots & \\
  0 &  \dots & 0 & A
\end{pmatrix} \in \mathcal{M}_{2n}(\R)
\end{equation}
where $A=\begin{pmatrix}  1 & -1 \\ -1 & 1 \end{pmatrix}$, $A_i=\begin{pmatrix} 1 & b_i \\ b_i & 1\end{pmatrix}$
with $b_1, \dots, b_n \in [-1,1].$
 
For all $i \in  [2n]$, one has $\E \left[ U_i^2 \right] = \E \left[ X_i^2 \right]$ and for $i \ne j$, 
$\sigma_{ij}^X \le \sigma_{ij}^U$  hence for all $i \ne j$, $\d_U(i,j)^2 \le \d_X(i,j)^2$. 
From  Corollary~\ref{theorem:application_main_theorem}, we conclude that
\[
\E \left[ \left(\proj_{H_U^\perp}(f_{\beta})(U)+\frac{1}{\beta} \right)^2 \right]
\leq 
\E  \left[ \left(\proj_{H_X^\perp}(f_{\beta})(X)+\frac{1}{\beta} \right)^2 \right].
\]
Therefore, it remains to prove that
 \begin{equation} 
 	\label{equation:conjecture_particular_case}
 	\E \left[ \left(\proj_{H_X^\perp}(f_{\beta})(X)+\frac{1}{\beta} \right)^2 \right]
 	\leq 
 	\E \left[ \left(\proj_{H_S^\perp}(f_{\beta})(S)+\frac{1}{\beta} \right)^2  \right]
 \end{equation}
and this will be the goal of the rest of this section. To
this end, we will again use a method of interpolation along a
monotonic path. To preserve simple covariance matrices on it, the
following lemma will be useful.
\begin{lemma}
	\label{lem:distri}
 Let $f$ be a $C^1$ function on $\R^n$ such that for all $x\in \R^n,$ and for all $u\in \R$, $f(x+u\mathbbm{1}) =
 f(x)+u.$ Then in distribution,
\[ 
\proj_{H_S^\perp} (f(S))  \overset{\mathrm{d}}{=} 
 \proj_{H_{\sqrt{\frac{n}{n-1}}G}^\perp}(f \left( \sqrt{
            \frac{n}{n-1}} \ G \right) )
\]
where $G \sim {\mathcal{N}}(0, \mathrm{Id})$ is a standard Gaussian random vector. 
\end{lemma}
\begin{proof}
We start by noticing that in distribution,
\[
S \overset{\mathrm{d}}{=}
\sqrt{\dfrac{n}{n-1}} \ (G-\overline{G} \mathbbm{1}),
\]
where
$\overline{G}=\dfrac{1}{n} (G_1 + \ldots G_n).$ 
Since $f(x+u\mathbbm{1}) = f(x)+u$, one has for all $x \in \R^n$, all $u \in \R$ and all $\ell \in [n]$, 
$\partial_\ell f(x+u\mathbbm{1})=\partial_\ell f(x).$
Therefore, in distribution,
\[
f(S) \overset{\mathrm{d}}{=} f\left(\sqrt{\dfrac{n}{n-1}} G\right) - \sqrt{\dfrac{n}{n-1}} \overline{G}
\quad \mathrm{and} \quad
\partial_\ell f(S) \overset{\mathrm{d}}{=} \partial_\ell f\left(\sqrt{\dfrac{n}{n-1}} G\right).
\]
Using again that for any fixed $x \in \R^n$, the function $u \mapsto f(x + u \mathbbm{1}) - u$ is constant, we deduce that $\sum_{i=1}^n \partial_i f=1.$ With Lemma \ref{lemma:expression_projection}, we conclude that in distribution
\[
\begin{aligned}
 \proj_{H_S^\perp}(f(S)) & = f(S)-\sum_{l=1}^n S_l \, \E\left[ \partial_l f(S) \right] 
 \\
&  \overset{\mathrm{d}}{=} f \left( \sqrt{\frac{n}{n-1}} G \right) - 
\sum_{l=1}^n \sqrt{\frac{n}{n-1}}G_l \ \E\left[ \partial_l f \left(\sqrt{\frac{n}{n-1}} G \right)  \right]  
\\
& \overset{\mathrm{d}}{=} \proj_{H_{\sqrt{\frac{n}{n-1}}G}^\perp}(f ( \left( \sqrt{\frac{n}{n-1}} G \right)).
\end{aligned}
\]
\end{proof}

From Lemma \ref{lem:distri}, inequality \eqref{equation:conjecture_particular_case} is therefore equivalent to 
\begin{equation}
	\label{eq:tobeproven}
\E  \left[ \left(\proj_{H_X^\perp}(f_{\beta}(X))+\frac{1}{\beta} \right)^2 \right]
  \leq 
\E  \left[ \left(\proj_{H_{\sqrt{\frac{2n}{2n-1}}G}^\perp}(f_{\beta}\left(\sqrt{\frac{2n}{2n-1}}G \right))+\frac{1}{\beta} \right)^2 \right]. 
\end{equation}
Let $Y := \sqrt{\frac{2n}{2n-1} } \, G$ be a Gaussian random vector independent of $X$, we define as in the proof of Proposition \ref{proposition:interpolation} the path
$Z_t = \sqrt{1-t} \, X + \sqrt{t} \, Y.$
\\
Let $\tilde{f_{\beta}} := f_{\beta}+\frac{1}{\beta}$. Then 
$\proj_{H_{Z_t}^\perp}
(\tilde{f_\beta}(Z_t)) =
\proj_{H_{Z_t}^\perp}
(f_\beta (Z_t)) +\frac{1}{\beta}
$
and if
$\tilde{\varphi}$ denotes the function defined  by
$
\tilde{\varphi}(t) 
= 
\E \left[ \left( \proj_{H_{Z_t}^\perp}
    (\tilde{f_\beta}(Z_t)) \right)^2 \right]
$
on $[0,1]$,
we get from Proposition \ref{proposition:interpolation}
\[
\begin{aligned}
	\tilde{\varphi}'(t) = 
	\frac{1}{2} \sum_{i\neq j} M_{ij} 
	&
	\Big( \mathbb{E} \left[ \partial_i \tilde{f_{\beta}}(Z_t) \partial_j \tilde{f_{\beta}}(Z_t) \right] -
	\mathbb{E} \left[ \partial_i \tilde{f_{\beta}}(Z_t) \right] \mathbb{E}\left[ \partial_j \tilde{f_{\beta}}(Z_t)\right] 
	\\
	+ & \
	\mathbb{E}\left[ \proj_{H_{Z_t}^\perp}(\tilde{f_{\beta}}(Z_t))\partial^2_{ij} \tilde{f_{\beta}}(Z_t) \right] 
	\Big)
\end{aligned}
\]
where 
\[
\frac{1}{2} \, M_{ij} = \frac{1}{2} (\d_X(i,j)^2-\d_Y(i,j)^2) = \frac{1}{2}\left( \sigma_{ii}^X + \sigma_{jj}^X - \sigma_{ii}^Y - \sigma_{jj}^Y \right) + \sigma_{ij}^Y - \sigma_{ij}^X.
\]
From \eqref{eq:derivees}, we know that $\partial_i \tilde{f_{\beta}} = p_i$, $\partial_{ij}^2 \tilde{f_{\beta}} = - \beta p_i p_j$ for $i \ne j$ therefore, with the abuse of notation that $p_i = p_i(Z_t)$, we get
\begin{align*}
\tilde{\varphi}'(t)& = \frac{1}{2} \sum_{i \ne j} M_{ij}
\left(
\E \left[  p_i p_j \right] - \E \left[ p_i \right] \E \left[ p_j \right] - \E \left[ \big(\beta \, \proj_{H_{Z_t}^\perp} (f_\beta (Z_t)) + 1 \big) p_i p_j \right]
\right)
\\
 & = \frac{1}{2} \sum_{i \ne j } (-M_{ij}) \left( \E \left[ p_i \right]  \E \left[ p_j \right] 
 + \beta  \E  \left[ \proj_{H_{Z_t}^\perp} (f_\beta (Z_t))  p_i p_j \right]  \right).
\end{align*}	
Recalling the definition of $\Sigma_X$ from \eqref{def_X} and of $\Sigma_Y = \frac{2n}{2n-1} \mathrm{Id}$, we get that 
\[
- \frac{1}{2} \, M =
\begin{pmatrix}
	\boxed{
		\begin{matrix}
			0 & -\frac{2(n-1)}{2n-1} \\
			-\frac{2(n-1)}{2n-1} & 0
		\end{matrix}
	}
	&
	\begin{matrix}
		\frac{1}{2n-1} & \frac{1}{2n-1} \\
		\frac{1}{2n-1} & \frac{1}{2n-1}
	\end{matrix}
	&
	\cdots
	&
	\begin{matrix}
		\frac{1}{2n-1} & \frac{1}{2n-1} \\
		\frac{1}{2n-1} & \frac{1}{2n-1}
	\end{matrix}
	\\[1em]
	\begin{matrix}
		\frac{1}{2n-1} & \frac{1}{2n-1} \\
		\frac{1}{2n-1} & \frac{1}{2n-1}
	\end{matrix}
	&
	\boxed{
		\begin{matrix}
			0 & -\frac{2(n-1)}{2n-1} \\
			-\frac{2(n-1)}{2n-1} & 0
		\end{matrix}
	}
	&
	\cdots
	&
	\begin{matrix}
		\frac{1}{2n-1} & \frac{1}{2n-1} \\
		\frac{1}{2n-1} & \frac{1}{2n-1}
	\end{matrix}
	\\[1em]
	\vdots & \vdots & \ddots & \vdots \\[1em]
	\begin{matrix}
		\frac{1}{2n-1} & \frac{1}{2n-1} \\
		\frac{1}{2n-1} & \frac{1}{2n-1}
	\end{matrix}
	&
	\begin{matrix}
		\frac{1}{2n-1} & \frac{1}{2n-1} \\
		\frac{1}{2n-1} & \frac{1}{2n-1}
	\end{matrix}
	&
	\cdots
	&
	\boxed{
		\begin{matrix}
			0 & -\frac{2(n-1)}{2n-1} \\
			-\frac{2(n-1)}{2n-1} & 0
		\end{matrix}
	}
\end{pmatrix}
\]
which means that $- \frac{1}{2} \, M = B + N$ where $B$ is the $n$-block diagonal formed with the $2 \times 2$ matrix 
$
\left(
\begin{matrix}
	0 & - \frac{2(n-1)}{2n-1}
	\\
	- \frac{2(n-1)}{2n-1} & 0
\end{matrix}
\right)
$
and $N$ is the $2n \times 2n$ matrix with constant entries equal to $1/(2n-1)$ outside the block matrices obtained for $B$.

 We will now exploit the symmetries of $Z_t$ inheritated by the symmetries of $X$ and $Y$. For any $k=1, \ldots, n$,
let $s_k \in \mathfrak{S}_{2n}$ be the symmetry defined by $s_k(1) = 2k-1$, $s_k(2) = 2k$ and $t_k \in \mathfrak{S}_{2n}$ be the transposition defined by $t_k(2k-1) = 2k$. For every $\sigma \in \mathfrak{S}_{2n}$, denote by $(Z_t)_\sigma$ the vector of coordinates $(Z_t)_{\sigma(i)}$. Therefore, the law of $Z_t$ is the same as the law of $(Z_t)_{\sigma}$ for every $\sigma$ obtained by compositions of some of the $s_k$'s and $t_\ell$'s. Moreover, for all $\sigma \in \mathfrak{S}_{2n},$
\[
f_\beta (Z_t) = f_\beta ((Z_t)_{\sigma}),
\quad
p_i \left((Z_t)_{\sigma} \right) = p_{\sigma(i)}(Z_t)
\]
and
\[
\proj_{H_{(Z_t)_\sigma}^\perp} (f_\beta ((Z_t)_\sigma)) = 
f_\beta ((Z_t)_\sigma) - \sum_{i=1}^{2n} (Z_t)_{\sigma(i)} \E \left[ p_i((Z_t)_\sigma) \right] = 
\proj_{H_{Z_t}^\perp} (f_\beta (Z_t)).
\]
From these symmetries, we deduce  that for all $i, j \in [2n]$,
\begin{eqnarray}
	\label{1}
 & \E p_i = \E p_1 = 1/2n,
\\
\label{2}
B_{ij} \,\E  \left[ \proj_{H_{Z_t}^\perp} (f_\beta (Z_t)) p_i p_j \right] & = B_{ij} \,\E  \left[ \proj_{H_{Z_t}^\perp} (f_\beta (Z_t)) p_1 p_2 \right],
\\
\label{3}
N_{ij} \, \E  \left[ \proj_{H_{Z_t}^\perp} (f_\beta (Z_t)) p_i p_j \right] & = N_{ij} \,\E  \left[ \proj_{H_{Z_t}^\perp} (f_\beta (Z_t)) p_1 p_3 \right].
\end{eqnarray}
Observe that the sum of the coefficients of each rows of the matrix $M$ is equal to zero. Thus,
$M \mathbbm{1} = 0$, and from \eqref{1}, we get that
\[
\sum_{1\leq i,j \leq 2 n} M_{ij}  \E \left[  p_i \right]  \E \left[ p_j \right] = 0.
\]
With \eqref{2} and \eqref{3}, we get
\begin{align*}
	\tilde{\varphi}'(t) = & \ \beta  \E  \left[ \proj_{H_{Z_t}^\perp} (f_\beta (Z_t)) 
\sum_{1\leq i,j \leq 2 n} \left(-\frac{1}{2} M_{ij} \right)  p_i p_j 
\right] 
\\
= & \ \beta  \E  \left[ \proj_{H_{Z_t}^\perp} (f_\beta (Z_t)) 
\sum_{1\leq i,j \leq 2 n}  (B_{ij} p_1 p_2 + N_{ij} p_1 p_3)
\right]
\\
= & \ \beta \E  \left[ \proj_{H_{Z_t}^\perp} (f_\beta (Z_t)) 
\left( - \frac{4n (n-1)}{2n-1} p_1 p_2  + \frac{4n^2 - 4n}{2n-1} p_1 p_3\right) \right]
\\
= & \ \frac{4n (n-1)}{2n-1} \beta \E  \left[ \proj_{H_{Z_t}^\perp} (f_\beta (Z_t)) \ p_1(p_3 - p_2) \right].
\end{align*}
Using again the invariance under $t_2 \circ t_1 \circ s_2$, 
$Z_{t}\overset{d}{=} (Z_4, Z_3, Z_2, Z_1, \dots)$, and we get that
\[
\tilde{\varphi}'(t)= \frac{4n(n-1)}{2n-1}\beta \E\left[ 
    \proj_{H_{Z_t}^\perp}( f_{\beta}(Z_t)) p_{4}(p_{2} - p_{3})
  \right].
\] 
Adding both equations shows that
\begin{equation}
	\label{eq:derivee_fin}
2\tilde{\varphi}'(t)= \frac{4n(n-1)}{2n-1}\beta \E\left[ 
\proj_{H_{Z_t}^\perp}( f_{\beta}(Z_t)) (p_{1}- p_{4})(p_{3} -
p_{2})\right]. 
\end{equation}
From \eqref{1} and Lemma \ref{lemma:expression_projection}, 
\[
\proj_{H_{Z_t}^\perp}( f_{\beta}(Z_t)) = f_\beta(Z_t) - \frac{1}{2n} \sum_{i=1}^{2n} (Z_t)_i
\]
hence, denoting the density of $Z_t$ by $h$, we find that $\tilde{\varphi}'(t)$ has the same sign as $I$, where
\begin{align*}
I & =   \E\left[
\proj_{H_{Z_t}^\perp}( f_{\beta}(Z_t)) (p_{1}- p_{4})(p_{3} -
p_{2})\right]
\\
& = 
\int_{\mathbb{R}^n} \left(f_{\beta}(z) - \frac{1}{2n}\sum_{i=1}^{2n}
z_i \right) \frac{(\e^{\beta z_1} - \e^{\beta
		z_4})(\e^{\beta z_3} - \e^{\beta z_2})}{\left( \sum_{i=1}^{2n}
	\e^{\beta z_i} \right)^2} h(z)\dz.
\end{align*}
We make the change of variable $z \mapsto \widehat{z}$  where $\widehat{z}_2 = z_3$, $\widehat{z}_3 = z_2$ and the other coordinates remaining unchanged hence
\[
I = \int_{\mathbb{R}^n} \left(f_{\beta}(z) - \frac{1}{2n}\sum_{i=1}^{2n}
z_i \right) \frac{(\e^{\beta z_1} - \e^{\beta
		z_4})(\e^{\beta z_2} - \e^{\beta z_3})}{\left( \sum_{i=1}^{2n}
	\e^{\beta z_i} \right)^2} h(\widehat{z})\dz .
\]
Adding the two expressions for $I$, we get 
\begin{equation}
	\label{eq:final}
2I = \int_{\mathbb{R}^n} \left(f_{\beta}(x) - \frac{1}{2n} \sum_{i=1}^{2n}
    z_i \right) \frac{(\e^{\beta z_1} - \e^{\beta z_4})(\e^{\beta z_3} - \e^{\beta z_2})}{\left( \sum_{i=1}^{2n}
      \e^{\beta z_i} \right)^2} \left( h(z) -h(\widehat{z})
  \right)\dz . 
\end{equation}
By the definition of $\Sigma_X$ from \eqref{def_X} and of $\Sigma_Y = \frac{2n}{2n-1} \mathrm{Id}$, we know that
the covariance matrix of $Z_t$ is
\[ \Sigma_{Z_t} = \begin{pmatrix}
  A_t & 0 & 0 & \dots & 0 \\
  0 & A_t & 0 & \dots & 0 \\
  0 & 0 & A_t & \dots & 0 \\
  \vdots & &  &\ddots & \\
  0 & 0 & \dots & 0 & A_t
\end{pmatrix} \in \mathcal{M}_{2n}(\mathbb{R})\]
where $A_t = \begin{pmatrix} a & b \\ b & a \end{pmatrix}$,  $a=1+\dfrac{t}{2n-1}>0$ and $b=t-1 <0$. 
Therefore
\[
A_t^{-1} = \frac{a}{a^2 - b^2} \, \mathrm{Id} - \frac{b}{a^2 - b^2} \begin{pmatrix} 0 & 1 \\ 1 & 0 \end{pmatrix}
\]
and the density $h$ has the following form
\[
  \begin{aligned}
    h(z)&=K_t\exp\left( -\frac{1}{2} \langle \Sigma_{Z_t}^{-1} z, z\rangle \right) \quad \textrm{ with } K_t>0
    \\
        &=K_t\exp\left(-\frac{1}{2}\times \frac{1}{a^2-b^2} \left( a\lvert z\rvert^2 -2b z_1 z_2 - 2b z_3 z_4 - \dots -2b x_{2n-1}x_{2n}\right) \right).
  \end{aligned}
\]
Since 
\[
\exp\left(u\right) - \exp\left(v \right)
= 
2 \exp\left(\frac{u+v}{2} \right) \sinh\left(\frac{u-v}{2}\right)
\]
we get that 
\[
  \begin{aligned}
    h(z)-h(\widehat{z})&=K_t\exp\left( -\frac{1}{2(a^2-b^2)} ( a\lvert z \rvert^2 -2bz_5z_6-\dots -2bz_{2n-1}z_{2n}) \right)\\
                       &\quad \times \exp \left(\frac{b}{a^2-b^2} \frac{z_1z_2+z_3z_4+z_1z_3+z_2z_4}{2} \right)  \\
                       &\quad \times 2\sinh\left( \frac{-b}{2(a^2-b^2)}(z_1-z_4)(z_3-z_2) \right).
  \end{aligned}
\]
Since $a^2 - b^2 > 0$,  we conclude that
\[
\frac{(\e^{\beta z_1} - \e^{\beta z_4})(\e^{\beta z_3} - \e^{\beta z_2})}{\left( \sum_{i=1}^{2n} \e^{\beta z_i} \right)^2}
(h(z)-h(\widehat{z})) 
\geq 0. 
\] 
It remains to remember that for
all $z\in \R^n$, one has
\[ f_{\beta}(z) - \frac{1}{2n}\sum_{i=1}^{2n} z_i \geqslant 0\] because
$f_{\beta}(z) \geqslant \max_{1 \le i \le 2n} z_i$.
Therefore, the integrand of $I$ in \eqref{eq:final} is positive and $I \ge 0$. Eventually  $\tilde{\varphi}'(t)\geq 0$
and $\tilde{\varphi}(0)\leq \tilde{\varphi}(1)$. This proves \eqref{eq:tobeproven} and concludes the proof of Theorem \ref{theorem:chevet_max_expectation_particular_case}.

{\bf Acknowledgements.} We thank Bernard Maurey for very fruitful discussions and Omer Friedland for sharing with us the idea of Proposition \ref{prop:Omer}. 
We are deeply grateful to the two anonymous referees for their insightful and constructive comments, which significantly improved the presentation of the paper. 

%\printbibliography
\bibliographystyle{plain} \bibliography{biblio}

\end{document}